\documentclass{conm-p-l}

\usepackage{amssymb,amsmath,amsthm,amscd,float,color,mathtools, pdflscape}
\allowdisplaybreaks

\usepackage{graphicx}

\usepackage{hyperref}

\hypersetup{
  colorlinks   = true, 
  urlcolor     = blue, 
  linkcolor    = blue, 
  citecolor   = red 
}

\usepackage[msc-links]{amsrefs}

\newtheorem{thm}{Theorem}

\newtheorem{lem}{Lemma}

\newtheorem{exa}{Example} 

\newtheorem{rem}{Remark}
\newtheorem{cor}{Corollary}

\DeclareMathOperator\Aut{Aut}
\DeclareMathOperator\h{\mathfrak h}

\newcommand\Z{\mathbb Z}
\newcommand\Q{\mathbb Q}
\newcommand\R{\mathbb R}
\newcommand\C{\mathbb C}

\def\H{\mathcal H}

\newcommand\F{\mathbb F}            
\newcommand\J{\mathcal J}
\def\O{\mathcal O}

\newcommand\G{\Gamma}                
\newcommand\e{\zeta}

\def\iso{\simeq}

\def\sig{\mbox{sig }}

\title{Minimal integral Weierstrass equations for genus 2 curves}
\author{L. Beshaj} 

\address{L. Beshaj \\
Department of Mathematics \\
The University of Texas at Austin\\
1 University Station C1200, Austin, Texas 78712}
\email{beshaj@math.utexas.edu }

\date{} 

\begin{document}

\begin{abstract}
We study the minimal Weierstrass equations for genus 2 curves defined over a ring of integers $\O_\F$. This is done via reduction theory and Julia invariant of binary sextics. We show that when the binary sextics has extra automorphisms this is usually easier to compute. Moreover, we show that when the curve is given in the standard form $y^2=f(x^2)$, where $f(x)$ is a monic polynomial,  $f(0)=1$ which is defined over $\O_\F$ then this form is reduced.   
\end{abstract}

\maketitle

\section{Introduction}
Let $\mathcal M_2$ be the moduli space classifying algebraic curves of genus 2.  
Using the classical theory of invariants of binary forms, J. Igusa (1960) constructed an arithmetic model of the moduli space $\mathcal M_2$. 
Given a moduli point ${\mathfrak p}\in \mathcal M_2 (\Q)$ there are basically two  main cases to get an equation of the curve defined over $\Q$  when  such an equation exists.  

If $\Aut ({\mathfrak p})$ has order 2, then one can use an algorithm of Mestre \cite{Me} to determine if there is a curve $\mathcal X$ defined over $\Q$ and construct its equation. 
 If $\Aut ({\mathfrak p})$ has order $ > 2$, then there always exists a curve $\mathcal X$ defined over $\Q$ and its equation can be found via the dihedral-invariants $(u, v)$ defined in  \cite{deg2}.  Such invariants determine uniquely the isomorphism class of a genus 2 curve with automorphism group isomorphic to the Klein 4-group $V_4$. 
The cases when the automorphism group of the curve is isomorphic to the dihedral group of order 8 or of order 12 correspond to the singular points of the $V_4$-locus in $\mathcal M_2$ and are treated differently; see \cite{beshaj-2} for such loci.  From  \cite{deg2}, and \cite{Sh3} we have a method which determines the equation of the curve when the curve has extra automorphisms. 
A more recent approach to recover an equation of a curve starting with any point ${\mathfrak p}\in \mathcal M_2$ regardless of the automorphism group can be found in the work of Malmendier and Shaska in \cite{malm-shaska}. In any case, for any number field $K$,  when a point ${\mathfrak p}\in \mathcal M_2 (K)$ is given we  can determine a genus 2 curve at worst defined over a quadratic extension of $K$. The reader interested in such computational routines can check \cite{alg-curves}.

Let us now assume that $\F$ is a number field and $\mathcal X$ be a genus 2 curve with extra automorphisms  defined over $\F$. Let $\O_\F$ be  the ring of integers  of $\F$.  Then, without any loss of generality we can assume that $\mathcal X$ is defined over $\O_\F$.        The height $\h (\mathcal X)$ of $\mathcal X$ over $\O_\F$ is defined in \cite{height-1}.  
We further assume that $F$ is a minimal field of definition  of  $\mathcal X$.
The focus of this paper is to find a twist $\mathcal X^\prime$ of $\mathcal X$ defined over $\O_\F$ such that the height $\h (\mathcal X^\prime)$ is minimal; see \cite{reduction} and \cite{height-1} for further details.

Reduction theory was introduced by Julia in \cite{julia} and has been revisited recently by \cite{cremona-red}, \cite{SC}, and \cite{reduction}. For every binary form there is a corresponding positive definite binary quadratic called the \textit{Julia quadratic}.  Since the Julia quadratic is positive definite, then there is a unique root in the upper half-plane $\H_2$. Hence, there is a map from the set of semistable binary forms to the upper half-plane, which is called the \textit{zero map} and denoted by $\e_0$. The binary form $f$ is called \textit{reduced} when $\e_0 (f)$ is in the fundamental domain of the modular group $\G :=SL_2 (\O_\F) / \{ \pm I \}$.  
For binary quadratics it is shown in \cite[Thm. 13]{reduction} that $f$ is reduced implies that $f$ has minimal height.  It is expected that this occurs under some mild condition in higher degrees as well. 

Hence, for any genus two curve $\mathcal X$ we write this curve in its Weierstrass form $y^2z^4=f(x, z)$over some algebraically closed field $k$, where $f(x, z)$ is a binary sextic. Finding a twist of $\mathcal X$ with smaller coefficients is equivalent to finding the reduction $\mbox{red } f$ as described in \cite{reduction}.  The main issue with this method is that determining $\mbox{red } f$ involves solving a system via Gr\"obener bases. There is also a numerical approach suggested in \cite{SC}, which of course it is open to numerical analysis. 

In this paper we investigate whether determining a Weierstrass equation with minimal height is easier in the case that the curve has extra involutions. Any curve with extra involutions can be written in  as $y^2 \,z^4=g(x^2, z^2)$, where $g(x, z)$ is a binary cubic form.  
We discover that if $g(x, z) \in \O_\F [x, z]$ and  $g(x^2, z^2)$ has minimal discriminant over $\O_\F$, then $y^2 \,z^4=g(x^2, z^2)$ is reduced.  

The paper is organized as follows. 
In section  two we describe briefly reduction theory of binary quintics and sextics, see \cite{beshaj-thesis} for more details.  The most delicate and difficult  part of reduction theory is computing the Julia quadratic. We show that computing the Julia quadratic for binary quintics and sextics in a direct way using the system given in ~ \cite[Eq.  4.13]{beshaj-thesis} is too difficult. Hence, we investigate alternative  methods,  considering separately totally real and totally complex binary forms. At the end of this section we give an example where we show how numerical computations can be used successfully in implementing a reduction algorithm.

In Section three we give a quick review of how for any binary form $f(x, z) \in \O_\F$ we can minimize the discriminant over $\O_\F$.  There is a detailed treatment of this in \cite{rachel}.  For minimizing the discriminant of genus 2 curves (over global fields) there is the more classical result of Liu \cite{Liu}. 
In section four we tailor the reduction for binary quintics and sextics and study how this can be performed when applied to forms with extra automorphisms. We show that the curves with extra automorphisms in the standard form as in $y^2 \,z^4=f(x, z)$ where $f(x, z) = x^6 +ax^4z^2+bx^2z^4+z^6$ are reduced over $\O_\F$ when defined over $\O_\F$ and have minimal discriminant.

In the last section we build a database of all such curves with height $\mathfrak h \leq 100$ defined over the integers.  There are 20 292 such curves (up to isomorphism over $\C$). For each height $1 \leq \mathfrak h \leq 100$ we also display the number of curves for that height.  The number of such curves with automorphism group $D_4$ and $D_6$ are also displayed. From these 20 292 curves we check if they are all of minimal height.  Of course not all of them are expected to have minimal discriminant.   We check how many of them have minimal height $\mathfrak h \leq 3$, while $\max \{ a, b \} > 3$.  We found 57 such cases, and as expected all of them do not have minimal discriminant over $\Z$.

\section{Reduction of binary quintics and sextics}
In this section we will define reduction theory for binary forms $f(x,z) \in \R[x,z]$. A generalization of reduction theory to binary forms defined over $\C$ is explained in details in \cite{beshaj-thesis}.  Let $f(x,z) \in \R[x,z]$ be a degree $n$ binary form  given as follows: 
\[ f(x,z) =a_0x^n+a_1x^{n-1}z+ \cdots +a_nz^n   \]
and suppose that $a_0 \neq 0$. Let the real roots of $f(x,z)$ be $\alpha_i$,   for $1 \leq i \leq r$ and the pair of complex roots $\beta_j$, $\bar \beta_j$ for $1 \leq j \leq s$, where $r+2s =n$. The form can be factored as 
\begin{equation}\label{deg-n-real-factored}
f(x,1) = \prod_{i=1}^r (x-\alpha_i) \cdot\prod_{i=1}^s (x-\beta_i)(x-\bar \beta_i).
\end{equation}
The ordered pair $(r, s)$ of numbers $r$ and $s$ is called the \textbf{signature} of the form $f$.  We associate to $f$ the two   quadratics $T_r (x, 1)$ and $S_s (x, 1)$   given by the formulas
\begin{equation}\label{def-T-S}
T_r (x, 1) =\sum_{i=1}^r t_i^2(x-\alpha_i)^2,   \quad \textit{and } \quad S_s (x, 1) =  \sum_{j=1}^s 2u_j^2(x-\beta_j)(x-\bar \beta_j ), 
\end{equation} 
%
%
where $t_i$, $u_j$ are   to be determined.   The quadratics $T$ and $S$ are positive definite binary quadratics with  discriminants  as follows 
\begin{equation}\label{lem-1}
\begin{split}
\Delta (T_r)  & = -4 \sum_{\substack{ i < j} }^r   t_i^2 t_j^2 \,    (\alpha_i - \alpha_j)^2,  \\
\Delta (S_s) & = -16 \left( \sum_{i < j} u_i^2 u_j^2 \left[  (a_i-a_j)^2 + (b_i^2+b_j^2) \right] + \sum_{j=1}^s u_j^4 b_j^2  \right).  
\end{split}
\end{equation}
where $\beta_i= a_i + I \cdot b_i$, for $i=1, \dots , s$. To each  binary form $f$ with signature $(r, s)$ we associate the quadratic  $Q_f$ which is defined as
\begin{equation}\label{Q_f}
Q_f (x, 1) =  T_r (x, 1) + S_s (x, 1). 
\end{equation} 
The quadratic $Q_f$ is  a positive definite quadratic with discriminant ${\mathfrak D_f}$  given  by
\begin{equation}\label{D-Q}
\begin{split}
{\mathfrak D_f} = & \Delta (T_r) + \Delta (S_s) - 8 \sum_{i, j} t_i^2 u_j^2 \left(  (a_i-a_j)^2 + b_j^2 \right).  
\end{split}
\end{equation}
Define the $\theta_0$ of a binary form as follows
\[ \theta_0 (f) = \frac{a_0^2\cdot |{\mathfrak D_f}|^{n/2}}{\prod_{i=1}^rt_i^2\, \prod_{j=1}^s u_j^4  }.\]
Consider $\theta_0(t_1, \dots, t_r, u_1, \dots, u_s)$  as a multivariable function in the variables $  t_i, u_j$ for $1 \leq i \leq r,  1\leq j\leq s$.  We would like to pick these variables such that $Q_f$ is a reduced quadratic. It is shown in \cite{beshaj-thesis} that  this is equivalent to $\theta_0(t_1, \dots, t_r, u_1, \dots, u_s)$ obtaining a minimal value.   The function  $\theta_0 : \R^{r+s} \to \R$  obtains a minimum at a unique point $(\bar t_1, \dots, \bar t_r,  \bar u_1, \dots, \bar u_s)$.  Julia in his thesis  \cite{julia} proves existence and Stoll, and Cremona prove uniqueness in \cite{SC}.

 Choosing $(\bar t_1, \dots, \bar t_r,  \bar u_1, \dots, \bar u_s)$  that   make $\theta_0$ minimal  gives a unique positive definite quadratic $Q_f (x,z)$.  We call this unique quadratic $Q_f (x,z)$ for such a choice of  $(\bar t_1, \dots, \bar t_r,  \bar u_1, \dots, \bar u_s)$  the \textbf{Julia quadratic} of $f(x,z)$,  denote it by $\J_f (x,z)$, and the quantity  $\theta_f:= \theta_0(\bar t_1, \dots, \bar t_r,  \bar u_1, \dots, \bar u_s)$ the \textbf{ Julia invariant}.     
%
%
The proof of the following lemma can be found in  \cite{beshaj-thesis}. 
\begin{lem}\label{theta-invariant}
Consider $GL_2 (\C)$ acting on $V_{n, \R}$. Then   
  $\theta$ is an   invariant of binary forms and   $\J$ is a  covariant  of order 2. 
\end{lem}
It is an open problem to express $\theta$ and $\J$ in terms of invariants and covariants of binary sextics.

To any form $f \in V_{n, \R}$ we associate   a positive definite quadratic $\J_f \in V_{2, \R}^+$  as showed above. In \cite{reduction} we show that positive definite  binary quadratic forms  are in one-to-one correspondence with  points in the upper half plane $\H_2$.  Hence, we have  the following map  $\zeta : V_{n, \R}   \mapsto V_{2, \R}  \to \H_2$, where 
\[ 
f  \mapsto  \J_f \mapsto  \xi (\J_f).
\]
where $\xi (\J_f)$  is the zero of the quadratic $\J_f$ in $\H_2$.  We call this map the \textbf{zero map} and denote it by $ \zeta (f) :=\xi(\J_f)$.     The map $\zeta : V_{n, \R} \to \H_2$ is  $\G$-equivariant, i.e. for any matrix $M$  in the modular group we have $\zeta(f^M) = \zeta(f)^{M^{-1}}$, see \cite{beshaj-thesis} for the proof.  A binary form $f \in V_{n, \R}$ is \textbf{reduced} if $ \zeta (f) \in \mathcal F_2$, where $\mathcal F_2$ is the classical fundamental domain.   

If the binary form is not reduced, i.e. $ \zeta (f) \notin \mathcal F_2$ let $M \in \G$ be such that $\zeta (f)^M \in  \mathcal F_2$. Since $\zeta$ is a $\G$-equivariant map then $f^{M^{-1}}$ is the reduced binary form which is  $\G$-equivalent to $f$. 

\subsection{Binary quintics}
Consider the case when   $f$ is a binary quintic with $\sig(f) = (5, 0)$.  Computations for binary forms with  $\sig(f) = (3, 1)$ and $\sig(f) = (1, 2)$ are similar and are treated in \cite{beshaj-thesis}. Let $f$ be  given as follows:
\begin{equation}\label{binary-quintc}
f(x, z)= a_0x^5+a_1x^4z+a_2x^3z^2+a_3x^2z^3+a_4xz^4+a_5z^5
\end{equation}
 and $\alpha_1, \dots, \alpha_5$ the roots of $f(x, 1)$.  We associate to $f$  the form 
\begin{equation}\label{eq-2}
 Q(x, z) = \sum_{i=1}^5 t_i^2(x-\alpha_i  z)^2  
 \end{equation}
which is a positive definite quadratic with  discriminant   
$  {\mathfrak D_f}= \sum_{\substack{ 1=i < j =5 }} u_i \, u_j \,  \alpha_{i, j}$, 
where   $u_i= t_i^2$ and  for $i<j$ we set  $\alpha_{i, j} = (\alpha_i -\alpha_j)^2$.  The  system  as in \cite[Eq.  4.13]{beshaj-thesis} is
\begin{equation}\left\{\begin{matrix}\label{sys-2}
& \frac{n-2} n \sum_{i\neq j } u_i u_j \alpha_{i, j} = \frac 2 n \sum_{ \substack{l, k=1 \\ l \neq k}}^n  u_j u_k \alpha_{j,k}\\
& u_1 \cdot u_2 \cdots u_6 =1
\end{matrix}\right. \end{equation}
Hence, we get
\[\left\{
\begin{aligned} 
3\,u_{{1}}u_{{2}}\alpha_{{1,2}}+3\,u_{{1}}u_{{3}}\alpha_{{1,3}}+3\,u_{{1}}u_{{4}}\alpha_{{1,4}}+3\,u_{{1}}u_{{5}}\alpha_{{1,5}}-2\,u_{{2}}u_{{3}}\alpha_{{2,3}}-\\
 2\,u_{{2}}u_{{4}}\alpha_{{2,4}}-2\,u_{{2}}u_{{5}}\alpha_{{2,5}}-2\,u_{{3}}u_{{4}}\alpha_{{3,4}}-2\,u_{{3}}u_{{5}}\alpha_{{3,5}}-2\,u_{{4}}u_{{5}}\alpha_{{4,5}}  & =0\\
  3\,u_{{1}}u_{{2}}\alpha_{{1,2}}-2\,u_{{1}}u_{{3}}\alpha_{{1,3}}-2\,u_{{1}}u_{{4}}\alpha_{{1,4}}-2\,u_{{1}}u_{{5}}\alpha_{{1,5}}+3\,u_{{2}}u_{{3}}\alpha_{{2,3}}+\\
3\,u_{{2}}u_{{4}}\alpha_{{2,4}}+3\,u_{{2}}u_{{5}}\alpha_{{2,5}}-2\,u_{{3}}u_{{4}}\alpha_{{3,4}}-2\,u_{{3}}u_{{5}}\alpha_{{3,5}}-2\,u_{{4}}u_{{5}}\alpha_{{4,5}}   & =0\\
 -2\,u_{{1}}u_{{2}}\alpha_{{1,2}}+3\,u_{{1}}u_{{3}}\alpha_{{1,3}}-2\,u_{{1}}u_{{4}}\alpha_{{1,4}}-2\,u_{{1}}u_{{5}}\alpha_{{1,5}}+3\,u_{{2}}u_{{3}}\alpha_{{2,3}}-\\
 2\,u_{{2}}u_{{4}}\alpha_{{2,4}}-2\,u_{{2}}u_{{5}}\alpha_{{2,5}}+3\,u_{{3}}u_{{4}}\alpha_{{3,4}}+3\,u_{{3}}u_{{5}}\alpha_{{3,5}}-2\,u_{{4}}u_{{5}}\alpha_{{4,5}}   & =0\\
 -2\,u_{{1}}u_{{2}}\alpha_{{1,2}}-2\,u_{{1}}u_{{3}}\alpha_{{1,3}}+3\,u_{{1}}u_{{4}}\alpha_{{1,4}}-2\,u_{{1}}u_{{5}}\alpha_{{1,5}}-2\,u_{{2}}u_{{3}}\alpha_{{2,3}}+\\
 3\,u_{{2}}u_{{4}}\alpha_{{2,4}}-2\,u_{{2}}u_{{5}}\alpha_{{2,5}}+3\,u_{{3}}u_{{4}}\alpha_{{3,4}}-2\,u_{{3}}u_{{5}}\alpha_{{3,5}}+3\,u_{{4}}u_{{5}}\alpha_{{4,5}}  & =0\\
-2\,u_{{1}}u_{{2}}\alpha_{{1,2}}-2\,u_{{1}}u_{{3}}\alpha_{{1,3}}-2\,u_{{1}}u_{{4}}\alpha_{{1,4}}+3\,u_{{1}}u_{{5}}\alpha_{{1,5}}-2\,u_{{2}}u_{{3}}\alpha_{{2,3}}-\\
 2\,u_{{2}}u_{{4}}\alpha_{{2,4}}+3\,u_{{2}}u_{{5}}\alpha_{{2,5}}-2\,u_{{3}}u_{{4}}\alpha_{{3,4}}+3\,u_{{3}}u_{{5}}\alpha_{{3,5}}+3\,u_{{4}}u_{{5}}\alpha_{{4,5}}  & =0\\
u_{{1}}u_{{2}}u_{{3}}u_{{4}}u_{{5}}-1 & =0
 \end{aligned} 
 \right.
\]
We want to solve the above system for $u_1, \dots, u_5$.   
\begin{lem} The degree of the field extension $[k(u_1, \dots, u_5):k(\alpha_{i, j})]=5$.
\end{lem}
 The proof is computational using Maple; see \cite{beshaj-thesis} for further details. 
%
%
%
\begin{rem}
For any given 5-tuple $(\alpha_1, \dots, \alpha_5)$ we have a unique positive real solution $(u_1, u_2, u_3, u_4, u_5)$.  Hence, as expected the coefficients of the Julia quadratic are uniquely defined. 
\end{rem}
One can express invariants of quintics in terms of the root differences $\alpha_{i, j}$ and then eliminate $u_1, \dots , u_5$ in order to get syzygies among $u_1, \dots , u_5$ and invariants of quintics.  Solving the system of such syzygies for $u_1, \dots , u_5$ one expects to get a unique real solution $(u_1, \dots , u_5)$.  Such a tuple would determine the Julia quadratic for the form $f$ in terms of the invariants of the quintics.   This seems quite a difficult problem computationally.


\subsection{Binary sextics}
Consider the case when   $f$ is a binary sextic with $\sig(f) = (6, 0)$.  Computations for binary forms with  $\sig(f) = (4, 1)$, $\sig(f) = (2, 2)$ and $\sig(f) = (0, 3)$ are similar and are treated in \cite{beshaj-thesis}. Let $f$ be  given as follows:
\[f(x,z)= a_0x^6+a_1x^5z+a_2x^4z^2+a_3x^3z^3+a_4x^2z^4+a_5xz^5+a_6z^6\]
  and $\alpha_i,$ $i=1, \dots, 6$,  its roots when $z=1$.  Associate to it the    quadratic 
\[Q_f = \sum_{i=1}^6 t_i^2(x-\alpha_i)^2,\]
where the $t_i$'s are non zero real numbers. For convenience, as above we fix the following notation $u_i:= t_i^2 $ and $\alpha_{i,j}:= (\alpha_i -\alpha_j)^2 $.  We can determine    $u_1, \cdots, u_6$  by solving the  system given as follows
\[\left\{\begin{matrix}
& \frac{n-2} n \sum_{i\neq j } u_i u_j \alpha_{i, j} = \frac 2 n \sum_{ \substack{l, k=1 \\ l \neq k}}^n  u_j u_k \alpha_{j,k}\\
& u_1 \cdot u_2 \cdots u_6 =1
\end{matrix}\right. \]  
more explicitly
\[
\left\{\begin{aligned}
&\frac 2 3 \left( u_1u_2\alpha_{1,2}+ u_1u_3\alpha_{1,3}+u_1u_4\alpha_{1,4}+u_1u_5\alpha_{1,5}+u_1u_6\alpha_{1,6}\right) - \frac 1 3 \left(  u_2u_3\alpha_{2,3}  \right.\\
&\;     +  u_2u_4\alpha_{2,4} + u_2u_5\alpha_{2,5}+ u_2u_6\alpha_{2,6} + u_3u_4\alpha_{3,4} + u_3u_5\alpha_{3,5}+ u_3u_6\alpha_{3,6} \\
& \; \left.  + u_4u_5\alpha_{4,5} + u_4u_6\alpha_{4,6} + u_5u_6\alpha_{4,6} \right) =0\\
&\vdots\\
&\frac 2 3 \left( u_1u_6\alpha_{1,6}+ u_2u_6\alpha_{2,6}+u_3u_6\alpha_{3,6}+u_4u_6\alpha_{4,6}+u_5u_6\alpha_{5,6}\right) - \frac 1 3 \left(  u_1u_2\alpha_{1,2}   \right.\\
&\;  +  u_1u_3\alpha_{1,3} + u_1u_4\alpha_{1,4} + u_1u_5\alpha_{1,5}+ u_2u_3\alpha_{2,3} + u_2u_4\alpha_{2,4} + u_2u_5\alpha_{2,5}+ u_3u_4\alpha_{3,4}  \\
&\;   \left. + u_3u_5\alpha_{3,5} + u_4u_5\alpha_{4,5}  \right) =0\\
& u_1 \cdot u_2 \cdots u_6 =1
\end{aligned}
\right. 
\]  
We want to solve the above system for $u_1, \dots, u_6$.    The proof of the following lemma is computational using Maple.
\begin{lem} The degree of the field extension $[k(u_1, \dots, u_6):k(\alpha_{i, j})]=36$.
\end{lem}



On the other hand,  the invariants of a binary sextic are given in terms of the difference of the roots.
We   work with invariants $J_2, J_4, J_6$, and $J_{10}$ as in \cite{alg-curves}.     Explicit formulas of  $J_2, J_4, J_6$, and $J_{10}$ in terms of  $\alpha_{i,j}$, i.e.  $J_k = g_k(\alpha_{i, j})$, $i, j=1, \dots, 6$ , $i <j$, and $k=2, 4, 6, 10$,  are well known. 

We would like    to express    $u_1, \dots, u_6$, hence the coefficient of Julia's quadratic,  in terms of the invariants  $J_2, J_4, J_6$, and $J_{10}$.   Computationally this is too difficult since for $i, j=1, \dots, 6$ and $i <j$ we have 15 variables $\alpha_{i,j}$. Hence,  we go back to the substitution, $\alpha_{i,j}=(\alpha_j- \alpha_i)^2$, where $\alpha_j, \alpha_i$ are the roots of the binary sextic.  Without loss of generality  fix a coordinate, i.e. let $\alpha_1=0$, $\alpha_2=1$, and $\alpha_3=-1$. This way we are left with  the problem of eliminating $\alpha_4$, $\alpha_5$, and $\alpha_6$.  

Eliminating  $\alpha_4$, $\alpha_5$, and $\alpha_6$  we are coming down from the field $k(\alpha_1, \dots, \alpha_6)$ to $k(u_1, \dots, u_6)$.   Hence, we get   equations only in terms of  $u_1, \dots, u_6$, and $J_2, J_4, J_6, J_{10}$.  The next thing is to solve for  $u_1, \dots, u_6$. 
Computations are very large and involve Gr\"obner bases.

Note that  computing the Julia quadratic using the above system is too  difficult,  hence we develop new methods. We consider separately totally real forms, i.e. forms with all real roots,  and totally complex forms  i.e. forms with complex roots.

\subsection{Totally real and totally complex forms}
Consider first totally real forms.  Let $f \in V_{n, \R}$ such that $f$ has signature $(n, 0)$  given by 
\[ f(x,z) = a_n x^n + a_{n-1} x^{n-1} z + \cdots + a_1 x z^{n-1} + a_0 z^n \]
where $a_0, \dots , a_n$ are transcendentals.  Identify $a_0, \dots, a_n$ respectively with $1, \dots, n+1$. Then  the symmetric group  $S_{n+1}$ acts on $\R[a_0, \dots a_n ] [x, z]$ by permuting $a_i$'s.  For any permutation $\tau \in S_{n+1}$ and $f \in \R[a_0, \dots a_n ] [x, z]$ we denote by $\tau (f) = f^\tau$. Then, 
\[ f^\tau (x, z) = \tau (a_n) \, x^n + \tau (a_{n-1} ) \, x^{n-1} z + \dots + \tau (a_1) \, x z^{n-1} + \tau (a_0) \, z^n. \]
Define $G_f(x, z)$    as follows
\begin{equation}\label{G-covariant}
G_f(x, z)= \frac  {    x \cdot f_x (- f_z (x, z), f_x (x, z) )+ z \cdot f_z (- f_z (x, z), f_x (x, z) )    }    {n \, f(x, z) }.
\end{equation}
In \cite{SC} Stoll and Cremona was proved that  $G_f(x, y)$ is a  degree $d=(n-1)(n-2)$ homogenous polynomial  in $\R[a_0, \dots a_n ] [x, z]$ and   $\J_f (x, z) \, | \, G_f(x, z)$; see \cite{SC} for details.  Therefore, this polynomial can be used to reduce totally real binary forms.  

Note that,  for  $\sigma \in S_{n+1}$ we have   an involution  
\[
\sigma=
\left\{
\begin{split}
& (1,n+1)(2, n) \cdots  \left( \frac n 2 ,  \frac n 2 +2 \right),                                    &   \text{ if   $n$ is even }  \\
&  (1, n+1)(2, n) \cdots  \left(    \frac {n+1} 2,  \frac {n+3} 2   \right),     &   \text{ if   $n$ is odd. }
\end{split}
\right.
\]
In \cite{beshaj-thesis} is proved that the polynomial $G_f(x, z)$ satisfies the following. 
\begin{thm}\label{thm1-totally-real}    
Let $f\in V_{n, \R}$ with distinct roots,  $\sig (f) = (n, 0)$, and $G_f$ as above. Then     


i) $G_f(x, z)$ is a covariant of $f$ of degree $(n-1)$ and order $(n-1)(n-2)$.  

ii) $G_f(x, z)$ has a unique quadratic factor  over $\R$, which is   $\J_f$.  

iii)  $G_f^\sigma (x, z)= (-1)^{n-1} \, G_f (x, z)$.  Moreover, if $ G_f = \sum_{i=1}^d   g_i \, x^i z^{d-i}$,   then 
   \[g_i^\sigma = (-1)^{n-1} \, g_{d-i},\]
    for all $i=0, ..., d$.
\end{thm}
Hence, to each totally real form $f$ we associate the $G_f$-polynomial which can be computed easily since is defined in terms of the derivatives of the  form. Then,  from Thm.~\ref{thm1-totally-real}  this polynomial has a unique quadratic factor which  is the Julia quadratic associated to the given binary form $f$. 

 Next, we illustrate the theory for totally real binary quintics. 
 
\subsection{Totally real quintics}
 
Let $f \in V_{5, \R}$ be a totally real binary quintic $f = \sum_{i=0}^5 a_i x^i z^{5-i}$ with roots $\alpha_i$ for $i=1, \dots 5$.  In the notation of the previous section, we have $Q_f= T_5$.  
The discriminant of  $T_5$ in terms of the roots  $\alpha_i$  of the form  is given by the formula
%
\[ 
\begin{split}
 \Delta \left(  T_5  \right) & = - 4 t_1^2 \cdots t_5^2 \left(  \frac {(\alpha_1 - \alpha_2)^2} {t_3^2 t_4^2 t_5^2} + \frac {(\alpha_1 - \alpha_3)^2} {t_2^2 t_4^2 t_5^2} +  \frac {(\alpha_1 - \alpha_4)^2} {t_3^2 t_2^2 t_5^2} +  \frac {(\alpha_1 - \alpha_5)^2} {t_2^2 t_3^2 t_4^2 }  \right. \\
& +  \frac {(\alpha_2 - \alpha_3)^2} {t_1^2 t_4^2 t_5^2} + \frac {(\alpha_2 - \alpha_4)^2} {t_1^2 t_3^2 t_5^2} + \frac {(\alpha_2 - \alpha_5)^2} {t_1^2 t_3^2 t_4^2} +   \frac {(\alpha_3 - \alpha_4)^2} {t_1^2 t_2^2 t_5^2}\\
&\left.   + \frac {(\alpha_3 - \alpha_5)^2} {t_1^2 t_2^2 t_4^2} +  \frac {(\alpha_4 - \alpha_5)^2} {t_1^2 t_2^2 t_3^2}   \right). 
\end{split}
\]
%
 In this case $\sigma = (1,6)(2,5)(3, 4)$ which correspond to 
$ \sigma = (a_0, a_5) (a_1, a_4) (a_2, a_3)$.
Then   computing $G_f(x, z)$ as in Eq.~\eqref{G-covariant} we have
\begin{equation}\label{T5}
G_f(x, z)= c_{12} x^{12} + c_{11}x^{11} z + \cdots c_1 x z^{11} + c_0 z^{12} 
\end{equation}
where the coefficients are given as follows: 
\[\begin{split}
c_{12} & =  125\,a_1{a_5}^3-50\,a_2a_4{a_5}^2+15\, a_3a_4^2a_5-3\,a_4^4  \\
c_{11} & =    625\,a_0{a_5}^3+175\,a_1a_4{a_5}^2-100\,a_2a_3{a_5}^2-55\,a_2a_4^2a_5+60\,a_3^2a_4a_5-9\,a_3a_4^3   \\
c_{10} & =  1375 \,a_0a_4{a_5}^2-25\,a_1a_3{a_5}^2+65\,a_1a_4^2a_5-150\,{a_2}^2{a_5}^2-30\,a_2a_3a_4a_5\\
&-41\,a_2a_4^3 +60\,a_3^3a_5+3\,a_3^2a_4^2   \\
c_9 & =  875\,a_0a_3{a_5}^2+1025\,a_0a_4^2a_5-425\,a_1a_2{a_5}^2+30\,a_1a_3a_4a_5-33\,a_1a_4^3\\
& -130\,{a_2}^2a_4a_5  +160\,a_2a_3^2a_5  - 86 \,a_2a_3a_4^2+ 24\,a_3^3a_4  \\
c_8 & =  125\,a_0a_2{a_5}^2+1500\,a_0a_3a_4a_5+215\,a_0a_4^3-425\,a_1^2{a_5}^2-450\, a_1a_2a_4a_5\\
 & +175\,a_1a_3^2a_5 -  128\,a_1a_3a_4^2  +175\,{a_2}^2a_3a_5-97\,{a_2}^2a_4^2+8\,a_2a_3^2a_4+12\,a_3^4   \\
c_7 & =    -750\,a_0a_1{a_5}^2+500\,a_0a_2a_4a_5+775\,a_0a_3^2a_5+  430\,a_0a_3a_4^2- 530\,a_1^2a_4a_5\\
   &+310\,a_1a_2a_3a_5 -322\,a_1a_2a_4^2-61\,a_1a_3^2a_4+105\,{a_2}^3a_5-33\,{a_2}^2a_3a_4+ 32\,a_2a_3^3 \\
c_6 & =   -625\,a_0^2{a_5}^2-800\,a_0a_1a_4a_5+1200\,a_0a_2a_3a_5+30\,a_0a_2a_4^2+365\,a_0a_3^2a_4\\
& + 30\,a_1^2a_3a_5 -303\,a_1^2a_4^2 + 365\,a_1{a_2}^2a_5- 268\,a_1a_2a_3a_4+a_1a_3^3+{a_2}^3a_4+37\,{a_2}^2a_3^2   \\
 c_5 & =  -750\,a_0^2a_4a_5+500\,a_{{0}
}a_1a_3a_5-530\,a_0a_1a_4^2+ 775\,a_0{a
_2}^2a_5+310\,a_0a_2a_3a_4\\
& +105\,a_0{a_{{
3}}}^3 +430\,a_1^2a_2a_5-322\,a_1^2a_3a_{
4}-61\,a_1{a_2}^2a_4-33\,a_1a_2a_3^2 +
32\,{a_2}^3a_3  \\
c_4 & =  125\,a_0^2a_{
3}a_5-425\,a_0^2a_4^2+1500\,a_0a_1a_{2
}a_5-450\,a_0a_1a_3a_4+175\,a_0{a_2}^2a_
4\\
&+175\,a_0a_2a_3^2  +215\,a_1^3a_5-128\,
a_1^2a_2a_4-97\,a_1^2a_3^2+8\,a_1{
a_2}^2a_3+12\,{a_2}^4 \\
\end{split}
\]
\[\begin{split}
c_3 & = 875\,{a_
0}^2a_2a_5-425\,a_0^2a_3a_4+1025\,a_{{0}
}a_1^2a_5+30\,a_0a_1a_2a_4-130\,a_0a_{
1}a_3^2\\
&+160\,a_0{a_2}^2a_3 -33\,a_1^3a_
4-86\,a_1^2a_2a_3+24\,a_1{a_2}^3\\
c_2  & = 1375\,a_0^2a_1a_5-25\,a_0^2a_2a_4-150\,a_0^2a_3^2+65\,a_0a_1^2a_4-30\,a_0a_1a_2a_3+60\,a_0{a_2}^{
3}\\
&-41\,a_1^3a_3+3\,a_1^2{a_2}^2  \\
c_1 & =  625\,a_0^3a_5+175\,a_0^2a_1a_4 -100\,a_0^2a_2a_3-55\,a_0a_1^2a_3+60\,a_0a_1{a_2}^2-9\,a_1^3a_2  \\
c_0 & = 125\, a_0^3 a_4-50\, a_0^2 a_1 a_3 + 15\,a_0 a_1^2 a_2-3\, a_1^4 \\
\end{split}
\]
%
The  following is an immediate consequence of Thm. \ref{thm1-totally-real}. 
\begin{cor} Let $f \in V_{5, \R}$ with signature $(5, 0)$. Then  
$G_f^\sigma= G_f$ the above coefficients give a computational confirmation that $G_f^\sigma = G_f$ and $ c_i^\sigma = c_{5-i}$ for all $i = 1, \dots , 5$. 
\end{cor}
When evaluated at a specific binary form the above polynomial  $G_f$ has a unique quadratic factor which is the Julia quadratic and can be used for reduction.

Similar computations are done for totally real sextics. The polynomial $G_f(x, y)$   is displayed in \cite[Appendix A]{beshaj-thesis}.   Next we consider totally complex forms.

\subsection{Totally complex} 
For reduction of totally complex forms   the Julia quadratic can be determined using the system given in  \cite[Eq. 3.14]{beshaj-thesis}. For such forms  the degree of the system will drop to $\frac n 2$ and the computations are possible. 

Let $f$ be a  totally complex binary sextic factored as  follows 
\[  f(x, z)= (x^2+a_1xz+b_1z^2)(x^2+a_2xz+b_2z^2)(x^2+a_3xz+b_3z^2).\] 
Associate to $f$ the quadratic 
\[  Q(x, 1) = 2u_1^2 (x^2+a_1x+b_1)+2u_2^2(x^2+a_2x+b_2)+2u_3^2(x^2+a_3x+b_3) \] 
where the  $u_i$'s are real numbers that make $\theta_f$ minimal. To find $u_1$, $u_2$ and $u_3$ that satisfy this condition  we need to set up the system in  \cite[Eq. 3.14]{beshaj-thesis} and solve for the $u_i$'s.   Compute first the  discriminant  of the quadratic which  is   as follows 
\[\begin{split}
{\mathfrak D_f}=  &4\,{a_{{1}}}^{2}{u_{{1}}}^{4}+8\,a_{{1}}a_{{2}}{u_{{1}}}^{2}{u_{{2}}}^
{2}+8\,a_{{1}}a_{{3}}{u_{{1}}}^{2}{u_{{3}}}^{2}+4\,{a_{{2}}}^{2}{u_{{2
}}}^{4}+8\,a_{{2}}a_{{3}}{u_{{2}}}^{2}{u_{{3}}}^{2}\\
&+4\,{a_{{3}}}^{2}{u
_{{3}}}^{4}-16\,b_{{1}}{u_{{1}}}^{4}-16\,b_{{1}}{u_{{1}}}^{2}{u_{{2}}}
^{2}-16\,b_{{1}}{u_{{1}}}^{2}{u_{{3}}}^{2}-16\,b_{{2}}{u_{{1}}}^{2}{u_
{{2}}}^{2}\\
&-16\,b_{{2}}{u_{{2}}}^{4}-16\,b_{{2}}{u_{{2}}}^{2}{u_{{3}}}^
{2}-16\,b_{{3}}{u_{{1}}}^{2}{u_{{3}}}^{2}-16\,b_{{3}}{u_{{2}}}^{2}{u_{
{3}}}^{2}-16\,b_{{3}}{u_{{3}}}^{4}
\end{split}\]
Next, compute the partial derivatives of ${\mathfrak D_f}$ with respect to $u_1, u_2$, and $u_3$ and then set up the system.  This is done in Maple but we do not display the system here because is too big.  The system is given in terms of $u_i$'s,  $a_i$'s, $b_i$'s and $\lambda$, the Lagrange multiplier.  Solving for $\lambda$  we get
\begin{small}
\[ \lambda= {\frac {4(a_{{3}}{u_{{1}}}^{2}a_{{1}}+a_{{3}}{u_{{2}}}^{2}a_{{2}}+{u_
{{3}}}^{2}{a_{{3}}}^{2}-2\,{u_{{1}}}^{2}b_{{1}}-2\,{u_{{2}}}^{2}b_{{2}
}-2\,b_{{3}}{u_{{1}}}^{2}-2\,b_{{3}}{u_{{2}}}^{2}-4\,{u_{{3}}}^{2}b_{{
3}})}{{u_{{3}}}^{2}{u_{{1}}}^{4}{u_{{2}}}^{4}}}\]
 \end{small}
Substitute $\lambda$  as computed in the \cite[Eq.  4.13 ]{beshaj-thesis}  and add to this system the equation $Q(x,1)=0$.  Using this approach we can compute the point $\xi(f)$  in the upper half plane corresponding  to the Julia quadratic.    Eliminating all three of them we get a degree 8  polynomial 
\begin{equation}\label{S3}
G_f = \sum_{i=0}^8 c_i x^iz^{8-i},
\end{equation} 
with coefficients given in \cite[Appendix A]{beshaj-thesis}.
This degree 8 polynomial   has a unique quadratic factor which is the Julia quadratic. 

As a special case, consider the case when we let $b_1=b_2=b_3=1$.  This case is interesting since curves with such equations  correspond to genus two curves with extra automorphism. Assume the binary form  $f$  is given as follows 
\[f(x, z) = \prod_{i=1}^3 (x^{2} + a_ixz+z^2).\] 
The polynomial $G_f(x, z)$ associated to this binary form   has coefficients as follows
\[ \begin{split}
c_8 &= -c_0 = 3\, \left(a_{{2}}   -a_{{3}}\right)  \left( a_{{1}}-a_{{3}} \right) 
 \left( a_{{1}}-a_{{2}} \right)  \left( a_{{1}}a_{{2}}+a_{{1}}a_{{3}}+
a_{{2}}a_{{3}} \right) \\
c_7 &= - c_1= 3\, \left( a_{{2}} -a_{{3}} \right)  \left( a_{{1}}-a_{{3}} \right) 
 \left( a_{{1}}-a_{{2}} \right)  \left( 3\,a_{{1}}a_{{2}}a_{{3}}+8\,a_
{{1}}+8\,a_{{2}}+8\,a_{{3}} \right) \\
c_6&=-c_2=6\, \left(a_{{2}}  -a_{{3}} \right)  \left( a_{{1}}-a_{{3}} \right) 
 \left( a_{{1}}-a_{{2}} \right)  \left( 5\,a_{{1}}a_{{2}}+5\,a_{{1}}a_
{{3}}+5\,a_{{2}}a_{{3}}+24 \right) \\
c_5&=-c_3= 9\, \left( a_{{2}} -a_{{3}} \right)  \left( a_{{1}}-a_{{3}} \right) 
 \left( a_{{1}}-a_{{2}} \right)  \left( a_{{1}}a_{{2}}a_{{3}}+8\,a_{{1
}}+8\,a_{{2}}+8\,a_{{3}} \right) \\
c_4&=0.
\end{split}\]
 Note that $G_f(x, z)$ is a palindromic polynomial. 
\begin{rem}
If $f = \prod_{i=1}^n (x^2 + a_i x z +z^2)$, then $f$ is a palindromic form.  In this case, the Julia quadratic is a factor of $G_f$, where $G_f$ is a self-inverse form.  
\end{rem}

 

%
%
%
%
%



We give an explicit example that illustrates the theory and emphasizes how numerical computations can be used in implementing a reduction algorithm. 
 
\begin{exa} Let $f  \in V_{\R, 6}$ be a binary form given as follows:
\[f(x, z) =x^6 - 12x^5z + 96x^4z^2 - 458x^3z^3 + 1489x^2z^4 - 3014xz^5 + 3770z^6\]
Let $z=1$ then the binary form  $f$ can be factored as 
\[ f(x, 1) = (x- \alpha_1)(x-\bar \alpha_1)(x- \alpha_2)(x-\bar \alpha_2)(x- \alpha_3)(x-\bar \alpha_3) \] 
where $\alpha_1= 1+3I$,  $\alpha_2= 2+5I$, and $\alpha_3= 3+2I$.  Associate to it the quadratic 
\[\begin{split}
 S(x,1)&= 2u_1^2 (x- \alpha_1)(x-\bar \alpha_1)+ 2u_2^2(x- \alpha_2)(x-\bar \alpha_2)+2u_3^2(x- \alpha_3)(x-\bar \alpha_3) \\
 &=  2u_1^2(x^2-2x+10)+2u_2^2(x^2-4x+29)+2u_3^2(x^2-6x+13).
 \end{split} \] 
To find $u_1, u_2$ and $u_3$  we need to solve the  following system 
\[
\left\{   
\begin{split}
 & -4\lambda u_1^3u_2^4u_3^4-576u_1^3-1120u_1u_2^2-544u_1u_3^2 =0\\
&-4\lambda u_1^4u_2^3u_3^4-1120u_1^2u_2-1600u_2^3-960u_2u_3^2=0 \\
& -4\lambda u_1^4u_2^4u_3^3-544u_1^2u_3-960u_2^2u_3-256u_3^3 = 0\\
&     u_1^4 \cdot u_2^4 \cdot u_3^4 =1.
\end{split}
\right. 
\]
First solve the above  for $\lambda$.  We have 
\[ \lambda= -\frac{8(17u_1^2+30u_2^2+8u_3^2)}{u_3^2u_1^4u_2^4}.\]
and 
\[
\left\{   
\begin{split}
&144u_1^4+280u_1^2u_2^2-240u_2^2u_3^2-64u_3^4=0\\
&280u_1^2u_2^2-136u_1^2u_3^2+400u_2^4-64u_3^4=0\\
&u_1^4 \cdot u_2^4 \cdot u_3^4 =1
\end{split}
\right. 
\]
We want to find $\zeta(f)$ which satisfies the polynomial $S(x,1)$. Make the  substitution $u_i^2=w_i$, for $i=1, 2, 3$  and solve the following system 
\[
\left\{   
\begin{split}
&144w_1^2+280w_1w_2-240w_2w_3-64w_3^2=0\\
&280w_1w_2-136w_1w_3+400w_2^2-64w_3^2=0\\
& w_1^2 \cdot w_2^2 \cdot w_3^2 =1 \\
& 2w_1(x^2-2x+10)+2w_2(x^2-4x+29)+2w_3(x^2-6x+13)=0.
\end{split}
\right. 
\]
The $w_1, w_2$ and $w_3$ can be given as rational functions of $x$, where $x$ satisfies  
\[\begin{split}
G_f(x)&=   2169x^8-13194x^7-221877x^6+2519358x^5-9693795x^4+\\
&+10471902x^3+54904153x^2-176055146x+134002174 
\end{split}\] 
By Strum's theorem the above polynomial has exactly 6 real roots.   Hence, it has only one quadratic factor which is the Julia quadratic. The zero of the Julia  is the point 
\[ w= 2.12067656802142+  3.26692991594356 \,  i.\]
The  form is reduced if this root is in the fundamental   domain, therefore send $w$ to $w_0=w-2$.    The reduced binary form is as follows 
\[ f(x+2, 1)= x^6 + 36x^4 - 10x^3 + 325x^2 - 250x + 1250.\]
This is the reduced form in its $\G$-orbit. To find the $GL_2(\C)$-reduced form notice that  $1250= 2 \cdot 5^4$ and if we send $x$ to $5x$ and then factor out the greatest common factor we get
\[ f(x, 1)=   2x^6 + 2x^5 + 13x^4 + 2x^3 + 36x^2 +  25 \] 
%
%
%
%
%
Thus, the minimal integral form is  
\[ f(x, 1)=   2 x^6 + 2 x^5 + 13 x^4 + 2 x^3 + 36 x^2 +  25 \] 
and $\h (f) = 36$. 
\end{exa}


Note that the reduction algorithm can be made  rather efficient using numerical methods, i.e. given a binary form $f$ we can compute $\zeta(f)$ numerically and then reduce $f$ as explained above. But we are interested in giving a precise formula of the Julia's invariant or Julia's quadratic in terms of the invariants of the binary form. This is quite difficult and can be done for binary forms under some assumptions as we will show next.

\section{Julia quadratic of genus two curves with extra automorphisms} 

In this section we give the Julia quadratic of  curves with extra automorphism  in terms of the invariants of binary forms.
 
\subsection{Genus 2 curve with $\Aut( \mathcal X) \cong V_4$}
%
%
In this subsection  we focus on genus 2 curves $\mathcal X$ with $\Aut(\mathcal X) \iso V_4$ that  can be written as $y^2 =f(x)$ for $f$ totally real.  If the form $f$ is not defined over $\O_F$ then we find an equation of the curve  over $\O_F$ as explained in  \cite{alg-curves}. 

For a totally real binary form we can perform reduction using the polynomial $G(x, z)$ defined in Eq.~\eqref{G-covariant}.  Computations show that $G_f(x, z)$   is factored in three factors.  One has degree $2$, one degree $6$, and one degree $12$ as follows
\begin{small}\[\begin{split}
G_f(x, 1) = &1024\,u^6 \,  \left( 4\,{u}^{3}-{v}^{2} \right) ^{3} \cdot g_0(x, 1) \cdot g_1(x, 1) \cdot g_2(x,1)
\end{split}\]\end{small}
where
\begin{small}
\begin{equation}\label{G-cov-V_4}
\begin{split}
g_0(x)= & {x}^{2} - ({v}^{2} - 4\,{u}^{3}) \\
g_1(x)= &\left( -{u}^{2}-3\,v \right) {x}^{6}+\left( 36\,{u}^{3}-2\,{u}^{2}v-
18\,{v}^{2} \right) {x}^{5}\\
&+\left( -4\,{u}^{5}+180\,{u}^{3}v+{u}^{2}{
v}^{2}-45\,{v}^{3} \right) {x}^{4}\\
&+ \left( -480\,{u}^{6}-16\,{u}^{5}v+
360\,{u}^{3}{v}^{2}+4\,{u}^{2}{v}^{3}-60\,{v}^{4} \right) {x}^{3}+\\
&+ \left( 16\,{u}^{8}-720\,{u}^{6}v-8\,{u}^{5}{v}^{2}+360\,{u}^{3}{v}^{3
}+{u}^{2}{v}^{4}-45\,{v}^{5} \right) {x}^{2}+ \\
&+\left( 576\,{u}^{9}-32\,
{u}^{8}v-576\,{u}^{6}{v}^{2}+16\,{u}^{5}{v}^{3}+180\,{u}^{3}{v}^{4}-2
\,{u}^{2}{v}^{5}-18\,{v}^{6} \right) x+\\
&+64\,{u}^{11}+192\,{u}^{9}v-48\,
{u}^{8}{v}^{2}-144\,{u}^{6}{v}^{3}+12\,{u}^{5}{v}^{4}+36\,{u}^{3}{v}^{
5}-{u}^{2}{v}^{6}-3\,{v}^{7}
\end{split}
\end{equation}\end{small}
while we don't display $g_2(x)$.  Computing the  discriminant of all of these factors we get 
\[\begin{split}
 &\Delta_{g_0}= 4 (v^2-4u^3) \\
 &\Delta_{g_1}=  2^{36} \,  3^2 \, \left( {u}^{2}+162\,u+12\,v-2187 \right) ^{2} \left( 4
\,{u}^{3}-{v}^{2} \right) ^{10}{u}^{30}\\
& \Delta_{g_2} = 2^{156} \, 3^{18} \, \left( u-9
 \right) ^{6} \left( {u}^{2}+162\,u+12\,v-2187 \right) ^{4} \cdot \\
 & \qquad  \cdot \left( {u}
^{2}+18\,u-4\,v-27 \right) ^{14} \left( 4\,{u}^{3}-{v}^{2} \right) ^{
66}{u}^{132}
\end{split}\]
We proved that $G_f(x, 1)$ has a unique quadratic factor, i.e. if  $v^2-4u^3  <0$ then   the Julia quadratic associated to $f$ is
\[ \J_f= x^2 - (v^2-4u^3).\]
If $v^2-4u^3=0$ the Julia quadratic is not defined,  but that corresponds to the case when the $\Aut (\mathcal X) \geq V_4$.  If $ v^2-4u^3 >0$ then for a generic  binary form we can not determine the Julia quadratic precisely in terms of the invariants $u$ and $v$.

\subsection{Genus 2 curve with $\Aut( \mathcal X) \cong D_4$}


Let $\mathcal X$ be a genus two curve with automorphism $\Aut( \mathcal X) \cong D_4$. The set of Weierstrass points in this case is $W= \left\{ \pm 1, \pm \alpha, \pm \frac 1 \alpha  \right\}$  i.e. $\alpha \in \R$.  Therefore the binary form corresponding to  $\mathcal X$ is a totally real form.  Computing the polynomial $G_f(x, z)$  defined in Eq.~\eqref{G-covariant} for this curves we have
\[\begin{split}
G_f(x, 1)= -  &\left( 3\,s-5\,{x}^{4}-{x}^{2} \right) \cdot \\
&   \left( {z}^{8}
{s}^{3}-70\,{s}^{2}{x}^{4}+14\,{s}^{2}{x}^{2}+25\,s{x}^{
8}-10\,s{x}^{6}+37\,s{x}^{4}+12\,{x}^{6} \right) 
\end{split}\]
and its discriminant is as follows
\[ \Delta_G= -  2^{44}\, \, 3^{5} \, \, 5^{21}
 \, {s}^{29} \left( 60\,s+1 \right) ^{2}
 \left( 20\,s-21 \right) ^{6} \left( -1+4\,s \right) ^{14} \left( 5\,s
+1 \right) ^{4}.
\]
For a given binary form  the Julia quadratic will be the unique quadratic factor of $G_f$.  
 
\subsection{Genus 2 curve with $\Aut( \mathcal X) \cong D_6$}

In an analogue way with the previous subsection,  curves with automorphism group $D_6$ fall under the category of totally real form. The polynomial $G_f(x, z)$ for this form is 
\begin{equation}\label{G-autD_6}
\begin{split}
G_f(x, 1)= -972 \, x \, &\left( -{x}^{6}+w \right) \\
& \left( 8\,w{x}^{12}+12\,w{x}^{9}+{x}
^{9}+12\,w{x}^{6}+12\,{w}^{2}{x}^{3}+w{x}^{3}+8\,{w}^{3} \right) 
\end{split}
\end{equation}
and its discriminant is 
\[ \Delta_G=  2^{114} \, \, 3^{198}
 \cdot\,{w}^{48} \left( 4\,w-1 \right) ^{39} \left( 32\,w+1 \right) ^{6}.\]
Hence, for genus two curves with extra involution we can conclude the following.  
 
\begin{thm} Let $\mathcal X$ be a genus two curve with $\Aut(\mathcal X) >2$, affine equation $y^2=f(x)$, and $F$ its field of moduli. Then, the following are true
 
i) If $\Aut (C) \iso V_4$, then  the Julia quadratic is the unique quadratic factor of $G_f(x,z)$ as defined in Eq.~\eqref{G-cov-V_4}. 

ii) If $\Aut (C) \iso D_4$, then    the Julia quadratic is the unique quadratic factor of
 \[\begin{split}
G_f(x, 1)= -  &\left( 3\,s-5\,{x}^{4}-{x}^{2} \right) \cdot \\
&   \left( {z}^{8}
{s}^{3}-70\,{s}^{2}{x}^{4}+14\,{s}^{2}{x}^{2}+25\,s{x}^{
8}-10\,s{x}^{6}+37\,s{x}^{4}+12\,{x}^{6}\right) 
\end{split}\]
iii) If $\Aut (C) \iso D_6$, then  the Julia quadratic is the unique quadratic factor of
\[
\begin{split}
G_f(x, 1)= -972 \, x \, &\left( -{x}^{6}+w \right) \\
& \left( 8\,w{x}^{12}+12\,w{x}^{9}+{x}
^{9}+12\,w{x}^{6}+12\,{w}^{2}{x}^{3}+w{x}^{3}+8\,{w}^{3} \right) 
\end{split}
\]
 \end{thm}

Next we would like to determine minimal models of genus two curves with extra automorphism.   First we will explain how one can get genus two curves with extra automorphism  with minimal discriminant, given the equation of the curve.  

\subsection{Minimal discriminants}
 Minimal discriminants of genus two curves are discussed in detail in \cite{alg-curves}. Here we  discus  the case of genus two curves with extra automorphism. 

Recall that given  a binary form $f(x, z)$ and a matrix  $M =  \begin{pmatrix} a &  b \\ c & d\end{pmatrix}$, such that $M \in GL_2(K)$, we have  that $f^M = f(ax +
bz, cx +dz)$ has discriminant $\Delta(f^M ) = (\det M)^{d(d-1)} \Delta(f)$.  Hence, given  $\mathcal X$  a genus two  curve with equation $y^2=f(x)$ and  $M \in GL_2 (K)$ such that $\det M = \lambda$ we have that  $ \Delta (f^M) =   \lambda^{30} \,  \Delta (f)$.

%

To reduce the discriminant we factor $\Delta$ as a product of primes, say $\Delta = p_1^{\alpha_1} \cdots p_r^{\alpha_r}$,   and take $u$ to be the product of those powers of primes with exponents $\alpha_i \geq 30$.  
Then, the transformation $M =  \begin{pmatrix} 1/u &  0 \\ 0 & 1\end{pmatrix}$ will reduce the discriminant, see \cite{alg-curves} for more details.   The following lemma is proved in \cite{alg-curves}
\begin{lem}  A genus 2 curve $\mathcal X_g$ with integral equation
\[ y^2 = a_6 x^6 + \cdots +a_1 x + a_0 \]
has minimal discriminant  if $\mathfrak v (\Delta ) <  30$.
\end{lem}

Now let us consider the case of curves with extra automorphism. Let $\mathcal X$ be hyperelliptic curve with an extra automorphism of order $n\geq 2$.  Then, from \cite{sh-issac}  we know that the equation of $\mathcal X$ can be written as 
\[ y^2= f(x^n) \]
If $\mathcal X$ is given with such equation over $\Q$ and discriminant $\Delta$,  then for any transformation $\tau:  x \mapsto \frac x {u^{\frac 1 n}}$ would have $\mathcal X^\tau$ defined over $\Q$ and isomorphic to $\mathcal X$ over $\C$. Hence, $\mathcal X^\tau$ is a twist of $\mathcal X$ with discriminant 
\[ \Delta^\prime =  \frac 1 {u^{\frac {30} n}} \cdot \Delta \]
Hence, the following result is clear. 

\begin{lem}
Let $\mathcal X$ be a genus 2 curve with $\Aut (\mathcal X) \iso V_4$.  Assume that $\mathcal X$ has equation over $\Q$ as $y^2= f(x^2)$.  Then, $\mathcal X$ has minimal discriminant $\Delta$ if $\mathfrak v (\Delta ) <  15$.

If $\Aut (\mathcal X) \iso D_6$ and    $\mathcal X$ has equation over $\Q$ as $y^2= f(x^3)$.  Then, $\mathcal X$ has minimal discriminant $\Delta$ if $\mathfrak v (\Delta ) <  10$.
\end{lem}

\begin{proof}
Assume $\Delta (f) = p^\alpha \cdot N$, where $\alpha>30$,  for some prime $p$ and some integer $N$ such that $(p, N)=1$.  Let $M =  \begin{pmatrix} 1/p &  0 \\ 0 & 1\end{pmatrix}$.   Then,  the  discriminant  of the form $f^M$ is  
\[\Delta^\prime = \frac 1 {p^{30} } \, \cdot  p^\alpha \cdot N = p^{\alpha - 30} \cdot N.\]
In the case of hyperelliptic curves with equation $y^2= f(x^2)$  let  $  M =  \begin{pmatrix} \frac{1 }{\sqrt{p} }&  0 \\ 0 & 1\end{pmatrix} $ then
\[\Delta^\prime = \frac 1 {\sqrt{p}^{30} } \, \cdot  p^\alpha \cdot N = p^{\alpha - 15} \cdot N.\]
The same way we can prove it for curves with equation $y^2= f(x^3)$.
\end{proof}

Next we will determine minimal models of genus two curves with extra automorphism. We study only the  loci in $\mathcal M_2$ of dimension $\geq 1$, other cases are obvious.

\section{Minimal models of  curves with extra involutions}

In this section we will focus on curves $\mathcal X$ with extra automrphisms.    The following lemma gives a choice for the set of Weierstrass points for curves with extra automorphism;  see \cite{deg2} for the proof.

\begin{lem}\label{Weierstrass-points}
 Let $\mathcal X$ be a genus 2 curve defined over a field $k$ such that $\mbox{ char } k \neq 2$ and $W$ be the set of Weierstrass points.  Then the following hold:

i) If $\Aut (\mathcal X) \iso V_4$, then $W=\{  \pm \alpha, \pm \beta, \pm \frac{ 1}{\alpha \beta} \} $

ii) If $\Aut (\mathcal X) \iso D_4$, then  $W= \left\{ \pm 1, \pm \alpha, \pm \frac 1 \alpha  \right\}$. 

iii) If $\Aut (\mathcal X) \iso D_6$, then $W= \{1, \e_3, \e_3^2, \l, \l \e_3, \l \e_3^2   \}$, where $\l$ is a parameter and $\e_3$ is a primitive third root of unity.  
\end{lem}

For each of the three cases above we already know how to find an  equation of the curve over its field of moduli as shown in \cite{alg-curves} amongst other places.  In the next theorem we discus integral equations and their reducibility for curves $\mathcal X$ with $\Aut (\mathcal X) \iso V_4$.  



\begin{thm}\label{aut_V4}
Let ${\mathfrak p} \in \mathcal M_2 (\Q)$ be such that $\Aut({\mathfrak p}) \iso V_4$. There is a genus 2 curve  $\mathcal X$ corresponding to $\mathfrak p$ with equation 
$y^2\, z^4 =f(x^2, z^2)$,  where 
\begin{equation}\label{eq_form}
 f(x, z)= x^6 - s_1 x^4 z^2 + s_2 x^2 z^4 - z^6. 
\end{equation} 
If $f \in \Z [x, z]$, then $f(x, z)$ or $f(-z, x)$ is a reduced binary form. 
\end{thm}

\begin{proof}
From Lem.~\ref{Weierstrass-points} we have that the set of Weierstrass points for such curves  
is $W=\{  \pm \alpha, \pm \beta, \pm \frac{ 1}{\alpha \beta}  \} $.  The affine equation of the corresponding curve is
\begin{equation}\label{eq_1}
 y^2 = (x^2 - \alpha^2) (x^2- \beta^2) \left( x^2 -  \frac 1 {(\alpha \beta )^2} \right). \end{equation}
Hence, 
\[ y^2\, z^4= x^6 - s_1 x^4 z^2 + s_2 x^2 z^4 - z^6. \]
where $s_1$ and $s_2$ are the symmetric polynomials.  
This proves the first part.

We assume now that $s_1, s_2 \in \Z$.
%
Then, if $\alpha$ is a non-real root so is its conjugate $\bar \alpha$.   Suppose   that  $\alpha $ and $\beta$ are both  purely complex. Then,  $\frac{1 }{\alpha\beta}$ is real.  Geometrically this case is illustrated in Fig. \eqref{fig1}.
\begin{figure}[h!]
\centering
\includegraphics[width=4cm, height=4.5cm]{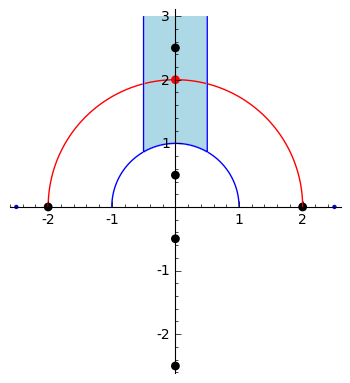}
\caption{The image of the zero map $\e (f)$ is the red dot}
\label{fig1}
\end{figure}

We are denoting with black dots the roots of the form.  In this case  the zero map $\e (f)$ will be in the "middle" of the half semicircle connecting $\pm \frac{1 }{\alpha \beta}$. Since, they are symmetric with respect to the y-axis this obviously will be in the y-axis, i.e. $\e (f)$ is purely complex. 
Next, assume  $\alpha $  is purely complex and $\beta$ is real.  Then $\frac {1}{\alpha \beta}$ is purely complex and the proof is as above.

Now, let us assume that  $\alpha$ and $\beta$ are complex roots with real and imaginary part nonzero.  Then, $\beta = \bar \alpha$ and  the set of Weierstrass points for the curve  is $\{ \pm \alpha, \pm \bar \alpha, \pm \frac 1 {\alpha \bar \alpha} \}$.     Then,   the centroid of the  rectangle  with vertices$\{ \pm \alpha, \pm \bar \alpha\}$  is  the origin $O$. Finding the zero map $\e (f)$ is equivalent to finding   the "middle" of the half semicircle connecting  the real roots $\pm \frac{1 }{\alpha \beta}$ which will be a point in the y-axis,  illustrated in Fig. \eqref{fig2}. 
\begin{figure}[h!]
\centering
\includegraphics[width=6cm, height=4cm]{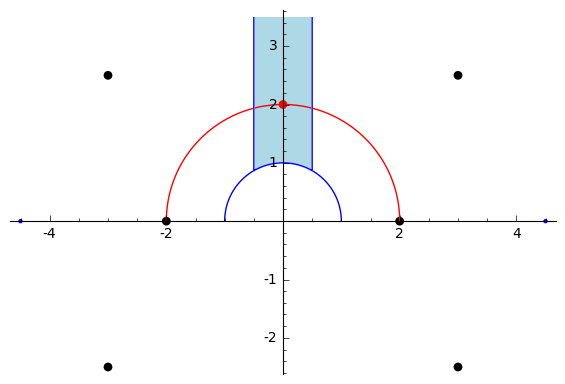}
\caption{The image of the zero map $\e (f)$ is the red dot}
\label{fig2}
\end{figure}


Above we proved that  $\e (F)$   is purely complex, i.e. $\e (F) = c i $ for some  $c\in \R^+$.  Then $f$ is reduced if and only if  $c \geq 1$.  Assume $f$ is not reduced.  Then,  there exists a binary form $g(x, z)$ such that  $g(x, z) = f(-z,x) = f(z, x)$. The form $g(x, z)$ will be reduced since $\e (g) = - \frac 1 {\e (f)} = \frac i c$. Hence,  either $f$ or $g$ (or possibly both, when $c=1$) will be reduced.

Lastly, if $\alpha$ and $\beta$ are both real the form  $f$ is totally real.  In \cite[Prop. 6.1 ]{beshaj-thesis} we proved that a superelliptic curve with such Weierstrass points is reduced in its orbit.   This completes the proof. 
\end{proof}


\begin{rem} After this proof was completed M. Stoll pointed out that since $\e (f)$  must be fixed by the extra involution $\sigma: (x, y) \to (-x, y)$. 
Notice  that $\e(f)$ is uniquely determined by the coefficients $s_1, s_2$.  Such coefficients are fixed by the extra involution $\sigma$.  Hence, $\e (f)$ is also fixed by such involution.  Thus,  $\e (f)$  is  purely complex.
\end{rem}

Notice that in general a binary form being reduced doesn't necessarily mean that it has minimal height; see \cite{reduction} for details.   And an integral model of the form given as in Eq.~\eqref{eq_form} is not necessarily of minimal height among all integral models. But it is  of minimal height among integral models defined by polynomials in $x^2$. This is proved in the following lemma.

\begin{lem} Let $\mathcal X$  be  a genus 2 curve    with $\Aut({\mathfrak p}) \iso V_4$ and  equation 
$y^2 =f(x^2)$,  where 
\begin{equation}\label{EQ_4}
 f(x)= x^6 - s_1 x^4  + s_2 x^2 - 1. 
\end{equation} 
Then, integral models of this form have minimal height  among integral models defined by polynomials in $x^2$, even up to twist. 
\end{lem}

\begin{proof}Let  $f(x)$ be integral given as in Eq.~\eqref{EQ_4}  and consider
another $\Q$-isomorphic integral model that is a polynomial in $x^2$. Any such model has coefficient vector as follows
 \[
( \lambda^6 \mu^2, \lambda^4 \mu^2 s_1, \lambda^2 \mu^2 s_2, \mu^2 )\]
for some rational $\lambda$ and $\mu$. Now, let us proceed prime by prime. Let $v$ be the valuation of $\lambda$ and let $w$ be that of
$\mu$. For the equation to have smaller height, we would need that one of the valuation
jumps
\[( 6 v + 2 w, 4 v + 2 w, 2 v + 2 w, 2 w )\]
to be  negative. But because the new model is also integral, we have $2 w \geq 0$ and $6 v + 2
w  \geq 0$. However, if those two inequalities hold, then all the jumps are
positive, so $4 v + 2 w$ and $2 v + 2 w$ as well. Therefore, our   integral model of  the form
 \begin{equation} 
 f(x)= x^6 - s_1 x^4  + s_2 x^2 - 1. 
\end{equation} 
has minimal height  among integral models defined by polynomials in $x^2$. 
\end{proof}

Note that lots of curves with geometric automorphism group $V_4$ and field of moduli $\Q$ do not admit a model over $\Q$ defined by a polynomial in $x^2$.  All of them descend, and in fact they even all descend to a hyperelliptic model instead of a cover of a conic. But they do not all admit that special form as given in Eq.~\eqref{eq_1}.  

The natural question is what are other additional conditions could force $f(x, z)$, where $f(x, z)$ represents  the equation of a curve with extra automorphism, to be of minimal absolute height?  We will elaborate more on this question in the next section. 

\section{Some heuristics for curves with extra involutions defined over $\Q$}

In \cite{alg-curves}  we display  a database of  genus 2 curves defined over $\Q$. The curves in the database are ordered based on their minimal absolute height, therefore they provide a perfect case for us to check how many of our curves are in that database.  In addition for each isomorphism class is given  a minimal equation over the field of moduli, the automorphism group,  and all the  twists.  All the computations involved in the database  are done based on the absolute  invariants $i_1, i_2, i_3$; see \cite{alg-curves} for details.   The database is  explained in more details in  \cite{alg-curves}. 

We have added to that  database  even the curves discussed here defined over $\Q$.  We have found all such curves of height $\h \leq 101$ defined over $\Q$.  The number of such curves for each height $\h$ is displayed in the following Table~\ref{table-1}. 

In the first column is the height of the curve, the second column contains the number of tuples $(1, 0, a, 0, b, 0, 1)$ which gives a genus 2 curves (i.e. $J_{10}\neq 0$). Not all such tuples give a new moduli point.   In the third column it is the number of such moduli points.  The fourth and fifth column contain the number of curves with automorphism group $D_4$ and $D_6$, and the last column contains the number of all points in the moduli space of height $\leq h$. 

\begin{small}

\begin{table}[b]
\caption{Number of curves with height $\h \leq 100$}
\begin{center}
\begin{tabular}{|c|c|c|c|c|c|| c|c|c|c|c|c| }
\hline
$\h$  & $J_{10}\neq 0$  &  in $\mathcal M_2$ & $D_4$ & $D_6$ & \# pts  &  $\h$  & $J_{10}\neq 0$  &  in $\mathcal M_2$  &   $D_4$     &   $D_6$   &  \# pts  \\
\hline
1 & 8 & 5 & 1 & 0 & 5  & 51 & 10607 & 205 & 2 & 0 & 5347 \\
2 & 24 & 9 & 2 & 0 & 14  & 52 & 11023 & 209 & 2 & 0 & 5556 \\
3 & 47 & 12 & 1 & 0 & 26 & 53 & 11447 & 213 & 2 & 0 & 5769 \\
4 & 79 & 17 & 2 & 0 & 43 & 54 & 11879 & 217 & 2 & 0 & 5986 \\
5 & 119 & 20 & 1 & 0 & 63 & 55 & 12319 & 221 & 2 & 0 & 6207 \\
6 & 167 & 25 & 2 & 0 & 88 & 56 & 12767 & 225 & 2 & 0 & 6432 \\
7 & 223 & 28 & 1 & 0 & 116 & 57 & 13223 & 229 & 2 & 0 & 6661 \\
8 & 287 & 33 & 2 & 0 & 149 & 58 & 13687 & 233 & 2 & 0 & 6894 \\
9 & 359 & 36 & 1 & 0 & 185 & 59 & 14159 & 237 & 2 & 0 & 7131 \\
10 & 439 & 41 & 2 & 0 & 226 & 60 & 14639 & 241 & 2 & 0 & 7372 \\
11 & 527 & 45 & 2 & 0 & 271 & 61 & 15127 & 245 & 2 & 0 & 7617 \\
12 & 623 & 49 & 2 & 0 & 320 & 62 & 15623 & 249 & 2 & 0 & 7866 \\
13 & 727 & 53 & 2 & 0 & 373 & 63 & 16127 & 253 & 2 & 0 & 8119 \\
14 & 839 & 57 & 2 & 0 & 430 & 64 & 16639 & 257 & 2 & 0 & 8376 \\
15 & 959 & 58 & 1 & 0 & 488 & 65 & 17159 & 261 & 2 & 0 & 8637 \\
16 & 1087 & 65 & 2 & 0 & 553 & 66 & 17687 & 265 & 2 & 0 & 8902 \\
17 & 1223 & 68 & 1 & 0 & 621 & 67 & 18223 & 269 & 2 & 0 & 9171 \\
18 & 1367 & 73 & 2 & 0 & 694 & 68 & 18767 & 273 & 2 & 0 & 9444 \\
19 & 1519 & 77 & 2 & 0 & 771 & 69 & 19319 & 277 & 2 & 0 & 9721 \\
20 & 1679 & 81 & 2 & 0 & 852 & 70 & 19879 & 281 & 2 & 0 & 10002 \\
21 & 1847 & 85 & 2 & 0 & 937 & 71 & 20447 & 285 & 2 & 0 & 10287 \\
22 & 2023 & 89 & 2 & 0 & 1026 & 72 & 21023 & 289 & 2 & 0 & 10576 \\
23 & 2207 & 93 & 2 & 0 & 1119 & 73 & 21607 & 293 & 2 & 0 & 10869 \\
24 & 2399 & 97 & 2 & 0 & 1216 & 74 & 22199 & 297 & 2 & 0 & 11166 \\
25 & 2599 & 101 & 2 & 0 & 1317 & 75 & 22799 & 301 & 2 & 0 & 11467 \\
26 & 2807 & 105 & 2 & 0 & 1422 & 76 & 23407 & 305 & 2 & 0 & 11772 \\
27 & 3023 & 109 & 2 & 0 & 1531 & 77 & 24023 & 309 & 2 & 0 & 12081 \\
28 & 3247 & 113 & 2 & 0 & 1644 & 78 & 24647 & 313 & 2 & 0 & 12394 \\
\hline 
\end{tabular}
\end{center}
\label{table-1}
\end{table}%

\begin{table}[t]
\begin{center}
\begin{tabular}{|c|c|c|c|c|c|| c|c|c|c|c|c| }
\hline
$\h$  & $J_{10}\neq 0$  &  in $\mathcal M_2$ & $D_4$ & $D_6$ & \# pts  &  $\h$  & $J_{10}\neq 0$  &  in $\mathcal M_2$  &   $D_4$     &   $D_6$   &  \# pts  \\
\hline
29 & 3479 & 117 & 2 & 0 & 1761 & 79 & 25279 & 317 & 2 & 1 & 12711 \\
30 & 3719 & 121 & 2 & 0 & 1882 & 80 & 25919 & 321 & 2 & 0 & 13032 \\
31 & 3967 & 125 & 2 & 0 & 2007 & 81 & 26567 & 325 & 2 & 0 & 13357 \\
32 & 4223 & 129 & 2 & 0 & 2136 & 82 & 27223 & 329 & 2 & 0 & 13686 \\
33 & 4487 & 133 & 2 & 0 & 2269 &  83 & 27887 & 333 & 2 & 1 & 14019 \\
34 & 4759 & 137 & 2 & 0 & 2406 &  84 & 28559 & 337 & 2 & 0 & 14356 \\
35 & 5039 & 141 & 2 & 0 & 2547 & 85 & 29239 & 341 & 2 & 0 & 14697 \\
36 & 5327 & 145 & 2 & 0 & 2692 &  86 & 29927 & 345 & 2 & 0 & 15042 \\
37 & 5623 & 149 & 2 & 0 & 2841 & 87 & 30623 & 349 & 2 & 0 & 15391 \\
38 & 5927 & 153 & 2 & 0 & 2994 & 88 & 31327 & 353 & 2 & 0 & 15744 \\
39 & 6239 & 157 & 2 & 0 & 3151 & 89 & 32039 & 357 & 2 & 0 & 16101 \\
40 & 6559 & 161 & 2 & 0 & 3312 & 90 & 32759 & 361 & 2 & 0 & 16462 \\
41 & 6887 & 165 & 2 & 0 & 3477 & 91 & 33487 & 365 & 2 & 0 & 16827 \\
42 & 7223 & 169 & 2 & 0 & 3646 & 92 & 34223 & 369 & 2 & 0 & 17196 \\
43 & 7567 & 173 & 2 & 0 & 3819 & 93 & 34967 & 373 & 2 & 0 & 17569 \\
44 & 7919 & 177 & 2 & 0 & 3996 & 94 & 35719 & 377 & 2 & 0 & 17946 \\
45 & 8279 & 181 & 2 & 0 & 4177 & 95 & 36479 & 381 & 2 & 0 & 18327 \\
46 & 8647 & 185 & 2 & 0 & 4362 & 96 & 37247 & 385 & 2 & 0 & 18712 \\
47 & 9023 & 189 & 2 & 0 & 4551 & 97 & 38023 & 389 & 2 & 0 & 19101 \\
48 & 9407 & 193 & 2 & 0 & 4744 & 98 & 38807 & 393 & 2 & 0 & 19494 \\
49 & 9799 & 197 & 2 & 0 & 4941 & 99 & 39599 & 397 & 2 & 0 & 19891 \\
50 & 10199 & 201 & 2 & 0 & 5142 & 100 & 40399 & 401 & 2 & 0 & 20292 \\
\hline 
\end{tabular}
\end{center}
\end{table}%

\end{small}


Some interesting questions that can be addressed analyzing Table~~\ref{table-1} are as follows.  How many genus two curves with extra involutions  are there with a fixed height $\h$? How many  isomorphism classes of genus two curves with extra involutions  are there for a fixed height $\h$? In other words, how many twists are for such curves with fixed  height? We intend to further explore some of these questions in further work.   

The main question that comes from the previous section was how many of these curves are of minimal absolute height. From 14523 = 20292 - 5769  binary forms of the form given in Eq.~\eqref{eq_form} we check how many of them have minimal absolute height $\leq 3$ even though $r := \min \{ | s_1 |, | s_2 | \} > 3$. Out of 14523 forms only   for  57 of them 
\[ r = \max \{ | s_1 |, | s_2 | \} \neq \h. \]
We display all such forms in the Table~~\ref{table-2}.  In the third column is the equation of the curve given the 7-tuple $(a_0, \dots , a_6)$ corresponding to the equation  
\[ y^2=\sum_{i=0}^6 a_i x^i = x^6 - s_1 x^4 + s_2 x^2 +1. \]  
 In the fifth column is the twist with height  $\h \leq 4$ which is isomorphic over $\overline \Q$ with the corresponding curve in the first column.  In the last column is given the automorphism group of the curve over $\overline \Q$.   

\begin{small}
\begin{table}
\caption{Curves which have twists with height $\leq 4$}
\begin{center}
\begin{tabular}{|c|c|c|c|c|c|}
\hline
\hline
\# & $r$ & $(1, 0, s_1, 0, s_2, 0, 1)$ & $\h$ & $(a_0, \dots , a_6)$ & $\mbox{Aut} (\mathfrak p)$  \\
\hline
1 & 7 & [1, 0, 1, 0, -7, 0, 1] & 3 & [1, -3, -1, -2, -1, -3, 1] & [4, 2]  \\ 
2 & 5 & [1, 0, 5, 0, 1, 0, 1] & 3 & [1, -1, 3, 2, 3, -1, 1] & [4, 2]  \\ 
3 & 17 & [1, 0, 15, 0, -17, 0, 1] & 2 & [1, -1/2, -1, -1, -1, -1/2, 1] & [4, 2]  \\ 
4 & 29 & [1, 0, -29, 0, -29, 0, 1] & 3 & [0, 1, 0, -3/2, 0, 1] & [8, 3]  \\ 
5 & 9 & [1, 0, 9, 0, 5, 0, 1] & 2 & [1, -1/2, 1, 1, 1, -1/2, 1] & [4, 2]  \\ 
6 & 41 & [1, 0, -25, 0, -41, 0, 1] & 3 & [1, -1/2, -3/2, 1, -3/2, -1/2, 1] & [4, 2]  \\ 
7 & 13 & [1, 0, 3, 0, -13, 0, 1] & 2 & [1, -2, -1, 0, -1, -2, 1] & [4, 2]  \\ 
8 & 51 & [1, 0, 51, 0, -45, 0, 1] & 3 & [1, 0, -1, -2/3, -1, 0, 1] & [4, 2]  \\ 
9 & 9 & [1, 0, 7, 0, -9, 0, 1] & 2 & [1, -1, -1, -2, -1, -1, 1] & [4, 2]  \\ 
10 & 19 & [1, 0, 19, 0, -13, 0, 1] & 2 & [1, 0, -1, -2, -1, 0, 1] & [4, 2]  \\ 
11 & 61 & [1, 0, 35, 0, -61, 0, 1] & 3 & [1, -2/3, -1, 2/3, -1, -2/3, 1] & [4, 2]  \\ 
12 & 7 & [1, 0, -7, 0, -7, 0, 1] & 4 & [1, 4, -3, 0, -3, -4, 1] & [8, 3]  \\ 
13 & 61 & [1, 0, 3, 0, -61, 0, 1] & 3 & [1, -2, -1, 3, -1, -2, 1] & [4, 2]  \\ 
14 & 17 & [1, 0, -1, 0, -17, 0, 1] & 3 & [1, -3, -1, 2, -1, -3, 1] & [4, 2]  \\ 
15 & 6 & [1, 0, 6, 0, 6, 0, 1] & 2 & [-1, 1, 1/2, 0, -1/2, -1, 1] & [8, 3]  \\ 
16 & 13 & [1, 0, -5, 0, -13, 0, 1] & 3 & [1, -1, -3, 2, -3, -1, 1] & [4, 2]  \\ 
17 & 29 & [1, 0, 19, 0, -29, 0, 1] & 3 & [1, -2/3, -1, 0, -1, -2/3, 1] & [4, 2]  \\ 
18 & 19 & [1, 0, 19, 0, 19, 0, 1] & 3 & [0, 1, 0, -3, 0, 1] & [8, 3]  \\ 
19 & 5 & [1, 0, -5, 0, -5, 0, 1] & 1 & [0, -1, 0, 0, 0, 1] & [48, 5]  \\ 
20 & 47 & [1, 0, 47, 0, 47, 0, 1] & 3 & [1, 0, -2/3, 0, -2/3, 0, 1] & [8, 3]  \\ 
21 & 39 & [1, 0, 39, 0, 23, 0, 1] & 2 & [1, -1/2, -1/2, 1, -1/2, -1/2, 1] & [4, 2]  \\ 
22 & 5 & [1, 0, 3, 0, -5, 0, 1] & 2 & [0, 1, -2, -2, -2, 1] & [2, 1]  \\ 
23 & 8 & [1, 0, 8, 0, 8, 0, 1] & 3 & [1, 1, 3/2, 0, 3/2, -1, 1] & [8, 3]  \\ 
24 & 19 & [1, 0, 19, 0, 11, 0, 1] & 2 & [1, -1/2, 0, 1, 0, -1/2, 1] & [4, 2]  \\ 
25 & 21 & [1, 0, -5, 0, -21, 0, 1] & 4 & [1, -4, -1, 4, -1, -4, 1] & [4, 2]  \\ 
26 & 53 & [1, 0, 43, 0, -53, 0, 1] & 3 & [1, -1/3, -1, 0, -1, -1/3, 1] & [4, 2]  \\ 
27 & 37 & [1, 0, 27, 0, -37, 0, 1] & 2 & [1, -1/2, -1, 0, -1, -1/2, 1] & [4, 2]  \\ 
28 & 35 & [1, 0, 35, 0, -29, 0, 1] & 1 & [1, 0, -1, -1, -1, 0, 1] & [4, 2]  \\ 
29 & 93 & [1, 0, 35, 0, -93, 0, 1] & 3 & [1, -1, -1, 3/2, -1, -1, 1] & [4, 2]  \\ 
30 & 21 & [1, 0, 11, 0, -21, 0, 1] & 1 & [1, -1, -1, 0, -1, -1, 1] & [4, 2]  \\ 
31 & 55 & [1, 0, 55, 0, 39, 0, 1] & 3 & [1, -1/3, -2/3, 2/3, -2/3, -1/3, 1] & [4, 2]  \\ 
32 & 77 & [1, 0, 51, 0, -77, 0, 1] & 2 & [1, -1/2, -1, 1/2, -1, -1/2, 1] & [4, 2]  \\ 
33 & 9 & [1, 0, -1, 0, -9, 0, 1] & 4 & [1, -4, -1, 0, -1, -4, 1] & [4, 2]  \\ 
34 & 37 & [1, 0, -37, 0, -37, 0, 1] & 3 & [1, 3/2, -1, 0, -1, -3/2, 1] & [8, 3]  \\ 
35 & 11 & [1, 0, 11, 0, 11, 0, 1] & 3 & [1, 2, 3, 0, 3, -2, 1] & [8, 3]  \\ 
36 & 29 & [1, 0, 3, 0, -29, 0, 1] & 2 & [1, -2, -1, 2, -1, -2, 1] & [4, 2]  \\ 
37 & 23 & [1, 0, 23, 0, 23, 0, 1] & 3 & [1, 0, -1/3, 0, -1/3, 0, 1] & [8, 3]  \\ 
38 & 11 & [1, 0, 5, 0, -11, 0, 1] & 3 & [1, -3/2, -1, -1, -1, -3/2, 1] & [4, 2]  \\ 
39 & 11 & [1, 0, 11, 0, 3, 0, 1] & 2 & [1, -1, 1, 2, 1, -1, 1] & [4, 2]  \\ 
40 & 25 & [1, 0, -9, 0, -25, 0, 1] & 2 & [1, -1, -2, 2, -2, -1, 1] & [4, 2]  \\ 
41 & 99 & [1, 0, 99, 0, -93, 0, 1] & 3 & [1, 0, -1, -1/3, -1, 0, 1] & [4, 2]  \\ 
42 & 15 & [1, 0, 9, 0, -15, 0, 1] & 3 & [1, -1, -1, -2/3, -1, -1, 1] & [4, 2]  \\ 
43 & 23 & [1, 0, 23, 0, 7, 0, 1] & 2 & [1, -1, 0, 2, 0, -1, 1] & [4, 2]  \\ 
44 & 45 & [1, 0, 19, 0, -45, 0, 1] & 1 & [1, -1, -1, 1, -1, -1, 1] & [4, 2]  \\ 
45 & 31 & [1, 0, 31, 0, 31, 0, 1] & 2 & [1, 0, -1/2, 0, -1/2, 0, 1] & [8, 3]  \\ 
46 & 9 & [1, 0, 9, 0, 9, 0, 1] & 3 & [-1, -3, -3, -2, 3, -3, 1] & [8, 3]  \\ 
47 & 25 & [1, 0, 7, 0, -25, 0, 1] & 3 & [1, -3/2, -1, 1, -1, -3/2, 1] & [4, 2]  \\ 
48 & 83 & [1, 0, 83, 0, -45, 0, 1] & 3 & [1, -1/2, -1, 3/2, -1, -1/2, 1] & [4, 2]  \\ 
49 & 13 & [1, 0, 13, 0, 9, 0, 1] & 3 & [1, -1/3, 1/3, 2/3, 1/3, -1/3, 1] & [4, 2]  \\ 
\hline 
\end{tabular}
\end{center}
\label{table-2}
\end{table}%

\begin{center}
\begin{tabular}{|c|c|c|c|c|c|}
\hline
\hline
%
%
%
50 & 25 & [1, 0, 23, 0, -25, 0, 1] & 3 & [1, -1/3, -1, -2/3, -1, -1/3, 1] & [4, 2]  \\ 
51 & 33 & [1, 0, 15, 0, -33, 0, 1] & 3 & [1, -1, -1, 2/3, -1, -1, 1] & [4, 2]  \\ 
52 & 27 & [1, 0, 27, 0, 19, 0, 1] & 3 & [1, -1/3, -1/3, 2/3, -1/3, -1/3, 1] & [4, 2]  \\ 
53 & 51 & [1, 0, 51, 0, -13, 0, 1] & 3 & [1, -1, -1, 3, -1, -1, 1] & [4, 2]  \\ 
54 & 13 & [1, 0, -13, 0, -13, 0, 1] & 1 & [0, 1, 0, -1, 0, 1] & [8, 3]  \\ 
55 & 27 & [1, 0, 27, 0, 27, 0, 1] & 3 & [1, 3, 3, 0, 3, -3, 1] & [8, 3]  \\ 
56 & 67 & [1, 0, 67, 0, -61, 0, 1] & 2 & [1, 0, -1, -1/2, -1, 0, 1] & [4, 2]  \\ 
57 & 33 & [1, 0, -33, 0, -33, 0, 1] & 3 & [1, 0, -3/2, 0, -3/2, 0, 1] & [8, 3]  \\ 
\hline 
\end{tabular}
\end{center}

\end{small}

\bigskip

There are a few questions which arise from Table~\ref{table-2}. First, can the curves of column five be obtained from reducing curves of column three? Secondly, are they in the same $\Gamma$-orbit as the curves from column 2? 


In response to this question, we found that twenty of the  curves displayed in Table~\ref{table-2}  can be reduced further using the reduction algorithm. They are displayed in Table~\ref{table-3}.  In the second column is displayed the curve from Table~\ref{table-2}, in the third column the curve obtained by the reduction algorithm and the last column the automorphism group of the curve.  Some of the reduced curves are isomorphic to the original curves over $\Q$.

\begin{small}

\begin{table}[h]
\caption{Curves which have twists with height $\leq 4$ and reduced curves. }
\begin{center}
\begin{tabular}{|c|c|c|c|}
\hline
\hline
case \# & $(s_1, s_2)$-curve      &   reduced curve    & Group \\
\hline 
57 & (1, 0, -33, 0, -33, 0, 1)      &   ( -2, 0, 3, 0, 3, 0,  -2)          		& [8, 3] \\
50 &  ( 1 ,   0, 23, 0, -25, 0, 1 )	&    ( 0 ,   3, 1, -6, 1, 3, 0)       & [4, 2] \\
21 &   ( 1 ,   0, 39, 0, 23, 0, 1 ) 	&    ( 2, 1, -1, -2, -1, 1, 2  )  	& [4, 2]  \\
16 &  ( 1 ,   0, -5, 0, -13, 0, 1 ) 	&    ( -1, -1, 3, 2, 3, -1, -1 )  			& [4, 2]  \\
20 &   ( 1 ,   0, 47, 0, 47, 0, 1 ) 	&    ( 3, 0, -2, 0, -2, 0, 3 )  		& [8, 3]  \\
39 &   ( 1 ,   0, 11, 0, 3, 0, 1 ) 	&    ( 1 ,   -1, 1, 2, 1, -1, 1  )  			& [4, 2]  \\
3 &   ( 1 ,   0, 15, 0, -17, 0, 1 ) 	&    ( 0, 2, 1, -4, 1, 2, 0 )  		& [4, 2]  \\
47 & ( 1 ,   0, 7, 0, -25, 0, 1 ) 	&    ( 1, -4, 3, 8, 3, -4, -1  )  		& [4, 2]  \\
40 & ( 1 ,   0, -9, 0, -25, 0, 1 ) 	&    ( -1, -1, 2, 2, 2, -1, -1 )  			& [4, 2]  \\
24 &  ( 1 ,   0, 19, 0, 11, 0, 1 ) 	&    ( 2, 1, 0, -2, 0, 1, 2  )  		& [4, 2]  \\
9 &  ( 1 ,   0, 7, 0, -9, 0, 1 ) 	&    ( 0, -1, 1, 2, 1, -1, 0 )  			& [4, 2]  \\
37 &  ( 1 ,   0, 23, 0, 23, 0, 1 ) 	&    ( 3, 0, -1, 0, -1, 0, 3 )  		& [8, 3]  \\
43 &  ( 1 ,   0, 23, 0, 7, 0, 1 ) 	&    ( 1 ,   -1, 0, 2, 0, -1, 1  )  			& [4, 2]  \\
31 &  ( 1 ,   0, 55, 0, 39, 0, 1 ) 	&    ( 3, 1, -2, -2, -2, 1, 3  )  & [4, 2]  \\
6 &  ( 1 ,   0, -25, 0, -41, 0, 1 )  &    ( -2, 1, 3, -2, 3, 1, -2  )  	& [4, 2]  \\
45 &  ( 1 ,   0, 31, 0, 31, 0, 1 ) 	&    (2, 0, -1, 0, -1, 0, 2 )  		& [8, 3]  \\
52 &  ( 1 ,   0, 27, 0, 19, 0, 1 ) 	&    ( 3, 1, -1, -2, -2, 1, 3 )  & [4, 2]  \\
14 &  ( 1 ,   0, -1, 0, -17, 0, 1 ) 	&    ( -1, 2, 3, -4, 3, 2, -1 )  			& [4, 2]  \\
51 &  ( 1 ,   0, 15, 0, -33, 0, 1 ) 	&    (-1, -6, 3, 12, 3, -6, 1  )  		& [4, 2]  \\
\hline 
\end{tabular}
\end{center}
\label{table-3}
\end{table}%

\end{small}

It is worth noting that in each case the reduction algorithm does find a curve with minimal absolute height.  It is also interesting to see that all 57 curves from Table~\ref{table-2} have one thing in common, their discriminant can be further reduced as explained in Section~(3.4).


%
%
%

We  believe that a generalization of Thm.~\ref{aut_V4} to  higher degree binary forms $f(x^2, y^2)$  and in more general for   forms  $f(x^n, y^n)$   is possible. Hopefully, this will be the focus of investigation of another paper. 

\bigskip

\noindent \textbf{Acknowledgments:} I would like to thank M. Stoll,  T. Shaska, and J. Sijsling for valuable comments and help.  Furthermore, I would like to thank the anonymous referees for all the comments they provided during the review. \\


\bibliographystyle{amsplain}

\bibliography{ref}{}

\end{document}